\documentclass[11pt]{article}

\usepackage{amsfonts,amsmath,amssymb,amsthm}
\usepackage[margin=1in]{geometry}
\usepackage{algorithm}
\usepackage[noend]{algpseudocode}

\newtheorem{theorem}{Theorem}[section]
\newtheorem{observation}{Observation}[section]
\newtheorem{proposition}[theorem]{Proposition}

\newtheorem{corollary}[theorem]{Corollary}
\newtheorem{definition}[theorem]{Definition}
\newtheorem{conjecture}[theorem]{Conjecture}

\begin{document}

\title{Tight bounds and conjectures for the isolation lemma}
\date{\today}

\author{
{\sc Vance Faber}$^{1}$
\and
{\sc David G. Harris}$^{2}$
}

\setcounter{footnote}{0}

\addtocounter{footnote}{1}
\footnotetext{IDA/Center for Computing Sciences, Bowie MD 20707. Email:\texttt{vance.faber@gmail.com}}

\addtocounter{footnote}{1}
\footnotetext{Department of Computer Science, University of Maryland, 
College Park, MD 20742. 
Research supported in part by NSF Awards CNS 1010789 and CCF 1422569.
Email: \texttt{davidgharris29@gmail.com}.}

\maketitle

\abstract{Given a hypergraph $H$ and a weight function $w: V \rightarrow \{1, \dots, M\}$ on its vertices, we say that $w$ is \emph{isolating} if there is exactly one edge of minimum weight $w(e) = \sum_{i \in e} w(i)$. The Isolation Lemma is a combinatorial principle introduced in Mulmuley et. al (1987) which gives a lower bound on the number of isolating weight functions. Mulmuley used this as the basis of a parallel algorithm for finding perfect graph matchings. It has a number of other applications to parallel algorithms and to reductions of general search problems to unique search problems (in which there are one or zero solutions).

The original bound given by Mulmuley et al. was recently improved by Ta-Shma (2015). In this paper, we show improved lower bounds on the number of isolating weight functions, and we conjecture that the extremal case is when $H$ consists of $n$ singleton edges.  We show that this conjecture holds in a number of special cases: when $H$ is a linear hypergraph or is 1-degenerate, or when $M = 2$. 

We also show that the conjecture holds asymptotically when $M \gg n \gg 1$ (the most relevant case for algorithmic applications).
}

\pagebreak

\section{Introduction}
Consider a hypergraph $H$ on $n$ vertices. We assign \emph{weights} $w$ to the vertices, which we regard as functions $w: [n] \rightarrow [M]$ (where we use the notation $[t] = \{1, \dots, t \}$). This weighting extends naturally to edges $e \in H$ by
$$
w(e) = \sum_{i \in e} w(i)
$$

We say that $e$ is a \emph{min-weight edge} (with respect to $w, H$) if for all edges $e' \in H$ we have $w(e') \geq w(e)$. Given a weight $w \in [M]^n$, we say that $w$ is \emph{isolating} (with respect to $H$) if there is exactly one min-weight edge; that is, there is an edge $e \in H$ with the property
$$
\forall e' \in H, e' \neq e \qquad w(e') > w(e)
$$

We refer to such an edge $e$ (if it exists) as \emph{isolated.}

Given any hypergraph $H$, we define
$$
Z(H, M) = \{ w \in [M]^n \mid \text{$w$ is isolating with respect to $H$} \}
$$

Our goal is to show lower bounds on the cardinality of $Z(H,M)$, which depend solely on $M$ and $n$ and are irrespective of $H$. 

Observe that when we are calculating the number of isolating weights, we may assume that $H$ is inclusion-free (i.e. there are no pair of edges $e, e' \in H, e \subsetneq e'$). We will make this assumption for the remainder of this paper. Also, by convention, if $H$ is the empty hypergraph (it contains no edges), then we say that every weight $w$ is isolating and define $Z(H,M) = [M]^n$. 

\subsection{Background}
The first lower bound on $|Z(H,M)|$, referred to as the \emph{Isolation Lemma}, was shown in \cite{mulmuley}, as the basis for a parallel algorithm to find a perfect matching in a graph. Other applications given in \cite{mulmuley} include parallel search algorithms and reduction of CLIQUE to UNIQUE-CLIQUE. The Isolation Lemma has also seen a number of uses in reducing search problems with an arbitrary number of possible solution to ``unique'' search problems (e.g. Unique-SAT), in which there is one or zero solutions. Two results in this vein which use the Isolation Lemma are reductions from NL (non-deterministic log-space) to UL (log-space with a unique solution) in \cite{wigderson, reinhardt}. In \cite{klivans-spielman}, a slightly generalized form of the Isolation Lemma was used for polynomial identity testing.

The usual algorithmic scenario can be summarized as follows. We have a hypergraph $H$ (which may not be known explicitly), which represents the space of possible solutions to some combinatorial problem. We wish to identify a unique edge $e \in H$ (a unique solution to the underlying problem). To do so, we select a random weight $w: [n] \rightarrow [M]$, where $M$ is a parameter to be chosen, and hope that $w$ has an isolated edge $e$. The probability that this occurs is $|Z(H,M)| / M^n$; thus, as long as $|Z(H,M)|$ is large compared to $M^n$, then this scheme has a good probability of succeeding in which case the overall algorithm will succeed as well. The ratio $|Z(H,M)|/M^n$ approaches $1$ as $M \rightarrow \infty$, and hence one can select $M$ sufficiently large to guarantee an arbitrarily-high success probability.

We emphasize that in such applications, typically we may \emph{choose} $M$, while the hypergraph $H$ is \emph{given} and we may have very little information about it.

The original work of \cite{mulmuley} showed a somewhat crude lower bound $|Z(H,M)| \geq M^n (1 - n/M)$. Notably, this lower bound is vacuous for $M \leq n$; however, because we may select $M$, this is not a problem algorithmically. For example,  in order to achieve $|Z(H,M)|/M^n = \Omega(1)$, we must select $M = \Omega(n)$.  In \cite{noam}, Ta-Shma improved this bound to $|Z(H,M)| \geq (M-1)^{n}$, which is strictly stronger than the bound of \cite{mulmuley}, and is non-vacuous even when $M < n$. We will review the proof by \cite{noam} in Section~\ref{bound-sec}. 

For most applications to computer science (where constant factors are irrelevant), these imprecise lower bounds on $|Z(H,M)|$ are perfectly adequate. It is nevertheless an interesting problem in extremal combinatorics to determine the tightest bound on $|Z(H,M)|$, even though this yields only minor computational savings.

We note that these algorithmic applications require a large supply of independent random bits. There has been another line of research in finding forms of the Isolation Lemma that use less randomness or can be made deterministic, such as \cite{arvind, srin, fenner, gurjar, svensson}. We do not investigate these issues in this paper.

\subsection{Overview}
In Section~\ref{bound-sec}, we discuss a generalization of the Isolation Lemma, and review the proof of \cite{noam}. We also state the main Conjecture~\ref{big-conj1} of our paper on the size of $Z(H,M)$, namely
$$
|Z(H,M)| \geq n \sum_{i=1}^{M-1} i^{n-1}
$$
for all $H$ and that this bound is tight.

We are able to show an improved bound on $|Z(H,M)|$ in Section~\ref{bound-sec}, namely
$$
|Z(H,M)| \geq 2 (M-1)^n - n \sum_{i=1}^{M-2} i^{n-1}
$$
When $M \gg n$, this is nearly optimal asymptotically, and improves significantly on the bound of \cite{noam} in all cases.

In Section~\ref{graph-transform-sec}, we show results which can be used to transfer the computation of $|Z(H,M)|$ to simpler graphs $H'$ with fewer vertices. These transformations show that Conjecture~\ref{big-conj1} holds for trees or 1-degenerate graphs. They also show that any minimal counterexample to Conjecture~\ref{big-conj1} must be connected and cannot contain vertices of degree zero or one.

In Section~\ref{m2sec}, we prove Conjecture~\ref{big-conj1} for the case $M = 2$.

In Section~\ref{linearsec}, we prove Conjecture~\ref{big-conj1} for linear hypergraphs.

In Section~\ref{sec:algorithmic}, we discuss asymptotics and algorithmic applications of these bounds.

In Section~\ref{sec:conclusion}, we conclude with some further open problems.

\section{Bounds and conjectures on $|Z(H,M)|$}
\label{bound-sec}
In nearly all application of the Isolation Lemma, the weights $w(i)$ are chosen as integers in the range $\{1, \dots, M \}$. However, the key to the Isolation Lemma is not the specific sizes of the weights, but their \emph{dynamic range.} We therefore introduce a slight generalization of the Isolation Lemma, in which we have a weight function $w: [n] \rightarrow W$, where $W \subseteq R_{>0}$ is a set of cardinality $|W| = M$.

An equivalent formulation (which will simplify some notations later on) is to have $w: [n] \rightarrow [M]$ and to specify an strictly increasing \emph{objective function} $f: [M] \rightarrow \mathbf R_{>0}$. We then define the weight of an edge by
$$
fw(e) = \sum_{i \in e} f( w(i))
$$

This type of generalized weight function is useful in some applications. For instance, \cite{narayanan} discusses a method of selecting a weight function in which there is an auxiliary function $g: [n] \rightarrow \mathbf R$ and the edge-weight is defined by $\sum_{i \in e} (w(i) + g(i))$. (This is slightly more general than allowing a single, vertex-independent, objective function). This generalized weight function will also be critical for some recursive proofs in this paper. 

We say as before that $e$ is isolated if $fw(e) < fw(e')$ for all $e' \neq e$. We may likewise define
$$
Z(H, M, f) = \{ w \in [M]^n \mid \text{$w$ is isolating with respect to $H, f$} \}
$$
When $f$ is the identity function, then $Z(H,M,f) = Z(H,M)$.

Our goal in this paper will be to show lower bounds on the cardinality of $Z(H,M,f)$, irrespective of $H, M, f$. Specifically, we define the quantity $Y(M,n)$ as
$$
Y(M,n) = \min_{H,f} |Z(H,M,f)|
$$
where $H$ ranges over all hypergraphs on $n$ vertices and $f$ ranges over all strictly increasing functions $f: [M] \rightarrow R_{>0}$.

We begin with two useful results which transform arbitrary weights into isolating weights.
\begin{proposition}[\cite{noam}]
\label{prop1}
Suppose that $w \in \{2, \dots, M\}^n$ and $e \in H$ is a min-weight edge for $fw$. Then $w - \chi_e$ is isolating for $f, H$, and $e$ is its isolated edge. 

(Here, $\chi_e$ is the characteristic function for $e$; that is, $\chi_e(v) = 1$ if $v \in e$ and $\chi_e(v) = 0$ otherwise)
\end{proposition}
\begin{proof}
Let $e' \in H, e' \neq e$ and let $w' = w - \chi_e$. Note that $e \cap e'$ is a \emph{strict} subset of $e$; for, if not, then this would imply $e \subsetneq e'$ which contradicts that $H$ is inclusion-free.

Then we have
\begin{align*}
fw'(e') -fw'(e) &= \sum_{ i \in e' - e} f(w'(i)) - \sum_{i \in e - e'} f(w'(i)) \\
&= \sum_{ i \in e' - e} f(w(i)) - \sum_{i \in e - e'} f(w(i) - 1) \\
&> \sum_{ i \in e' - e} f(w(i)) - \sum_{i \in e - e'} f(w(i)) \\
&\qquad \qquad \text{as $e - e' \neq \emptyset$ and $f$ is strictly increasing} \\
&= fw(e') - fw(e) \geq 0
\end{align*}
\end{proof}

Using Proposition~\ref{prop1}, Ta-Shma gave a simple lower bound on $|Z(H,M,f)|$:
\begin{proposition}[\cite{noam}] 
For all $M, H, f$ we have
$$
|Z(H,M,f)| \geq (M-1)^n
$$
\end{proposition}
\begin{proof}
We construct an injective map $\Psi$ from $\{2, \dots, M \}^n$ to $Z(H,M,f)$, as follows. Given any 
$w \in \{2, \dots, M \}^n$, arbitrarily select one min-weight edge $e$, and map $\Psi(w) = w - \chi_e$. By Proposition~\ref{prop1} the images of this map are all isolating. Also, this map is injective: given some $w \in \text{image}(\Psi)$, it has an isolated edge $e$ and its pre-image is $\Psi^{-1}(w) = w + \chi_e$.
\end{proof}

The next proposition is at the heart of our improvement over Ta-Shma's work:
\begin{proposition}
\label{prop1a}
Suppose that $w \in [M]^n$ and $e \in H$ is a min-weight edge for $f$. Suppose there is some $\ell \in [n]$ such that all min-weight edges contain $\ell$,  and that $w(i) \geq 2$ for $i \neq \ell$.

Then $w - \chi_{e - \{\ell \}}$ is isolating for $f, H$, and $e$ is its isolated edge.
\end{proposition}
\begin{proof}
Let $e' \in H, e' \neq e$ and let $w' = w - \chi_{e - \{\ell \}}$. There are two cases. First, suppose that $\ell \in e'$. Then
\begin{align*}
fw'(e') - fw(e) &= \sum_{ i \in e' - e} f(w'(i)) - \sum_{i \in e - e'} f(w'(i)) \\
&= \sum_{ i \in e' - e} f(w(i)) - \sum_{i \in e - e'} f(w(i) - 1) \\
&> \sum_{ i \in e' - e} f(w(i)) - \sum_{i \in e - e'} f(w(i)) = fw(e') - fw(e) \geq 0
\end{align*}

Next, suppose that $\ell \notin e'$. Then $fw(e') > fw(e)$ and so
\begin{align*}
fw'(e') - fw'(e) &= \sum_{ i \in e' - e} f(w'(i)) - \sum_{i \in e - e'} f(w'(i)) \\
&= \sum_{ i \in e' - e} f(w(i)) - f(w(\ell)) - \sum_{ i \in (e - \{ \ell \}) - e' } f(w(i)-1)  \\
&\geq \sum_{ i \in e' - e} f(w(i)) - f(w(\ell)) - \sum_{ i \in (e - \{ \ell \}) - e' } f(w(i)) \\
&= fw(e') - fw(e) > 0
\end{align*}

\end{proof}

\subsection{The conjectured extremal case: the singleton hypergraph}
We define the \emph{singleton hypergraph} $S_n$, which has vertex set $[n]$ and $n$ singleton edges $\{1 \}, \dots, \{ n \}$. We likewise define its complement graph $\bar S_n$, which has all $n$ edges of cardinality $n-1$.

\begin{observation}
For any $M,f$ we have
$$
|Z(S_n, M,f)| = |Z(\bar S_n, M, f)| = n \sum_{i=1}^{M-1} i^{n-1}
$$
\end{observation}
\begin{proof}
Any isolating weight for $S_n$ has the following form: one vertex $i$ is assigned weight $w(i) = j$, and the other vertices are assigned weights $> j$. 

Any isolating weight for $\bar S_n$ has the following form: one vertex $i$ is assigned weight $w(i) = j$, and the other vertices are assigned weights $< j$.
\end{proof}

We conjecture that this bound is tight. 
\begin{conjecture}
\label{big-conj1}
$$
Y(M,n) = n\sum_{i=0}^{M-1} i^{n-1}
$$
\end{conjecture}

One strategy that will be useful is to categorize weights in terms of their lowest value vertex. More formally, for any weight $w$, we define the $\emph{layer}$ of $w$ to be $$
L(w) = \min_{x \in [n]} w(x).
$$
For $j = 1, \dots, M$, we define $Z_j(H, M, f)$ to be the set of isolating weights $w$ with the property that $L(w) = j$. Similarly we define a universal lower bound $Y_j(M,n)$ such that $|Z_j(H,M,f)| \geq Y_j(M,n)$. 

Observe that $Z_1(S_n, M, f) = n (M-1)^{n-1}$ for any choice of $f$. We again conjecture that this bound is tight.
\begin{conjecture}
\label{big-conj2}
$$
Y_1(M,n) = n (M-1)^{n-1}
$$
\end{conjecture}

Although Conjecture~\ref{big-conj2} involves only $Y_1$, it implies bounds for all $Y_2, \dots, Y_M$.

\begin{proposition}
\label{prop2}
For all $M, n, j$ we have $Y_j(M,n) = Y_1(M-j+1,n)$.
\end{proposition}
\begin{proof}
For any $M,j$ define $W_{M,j}$ to be the set of weights $w \in [M]^n$ with $L(w) = j$.

Suppose $H,f$ satisfies $|Z_j(H,M,f)| = Y_j(M,n)$. Define the function $g:[M-j+1] \rightarrow \mathbf R_{>0}$ by $g(k) = f(k+j-1)$. For $w \in W_{M-j+1,j}$ we have $gw(e) = fw'(e)$ for all edges $e$, where $w' = w + (j-1)$. Also, note that the function mapping $w$ to $w' = w + (j+1)$ is a bijection from $W_{M - j + 1, 1}$ to $W_{M,j}$. So $|Z_j(H,M,f)| = |Z_1(H,M - j + 1,g)|$ and hence $Y_1(M - j + 1) \geq Y_j(M,n)$.

Suppose $H,f$ satisfies $|Z_1(H,M -j + 1,f)| = Y_1(M - j + 1,n)$. Since the image of $f$ is strictly positive, there is a real number $\alpha > 0$ such that $\alpha M < f(1)$. Define the function $g:[M] \rightarrow \mathbf R_{> 0}$ by
$$
g(k) = \begin{cases}
\alpha k & k < j \\
f(k - j + 1) & k \geq j
\end{cases}
$$

This function is strictly increasing and positive. For $w \in W_{M,j}$, note that $g w(e) = f w'(e)$ for all edges $e$, where $w' = w- j+1 $. Also, note that function mapping $w$ to $w' = w - j+1$ is a bijection from $W_{M,j}$ to $W_{M-j+1,1}$. So $|Z_1(H,M - j + 1,f)| = |Z_j(H,M,g)|$ and hence $Y_j(M,n) \geq Y_1(M+j+1,n)$.
\end{proof}

\begin{corollary}
\label{scor1}
We have $Y(M,n) \geq \sum_{i=1}^M Y_1(i,n)$.
\end{corollary}
\begin{proof}
Let $H, M, f$ be given. Then Proposition~\ref{prop2} gives
\begin{align*}
|Z(H, M, f)| = \sum_{j=1}^M |Z_j(H,M,f)| \geq \sum_{j=1}^M Y_j(M,n) = \sum_{j=1}^M Y_1(M-j+1,n).
\end{align*}
\end{proof}

\begin{corollary}
\label{conjprops}
Conjecture~\ref{big-conj2} implies Conjecture~\ref{big-conj1}.
\end{corollary}

We note that even if one is only interested in bounding $|Z(H,M)|$ (i.e. the case $f = \text{identity}$),  Proposition~\ref{prop2} requires bounds on $|Z(H,M',f')|$ for $f' \neq \text{identity}$. This is the main reason we need to consider the generalized $Z(H,M,f)$, instead of the simpler $Z(H,M)$, in this paper.

\subsection{An improved bound on $Y_1(M,n)$}

Although we cannot show Conjecture~\ref{big-conj2} in general, in Theorem~\ref{main-thm} we show a new lower bound on $Y_1(M,n)$, which can be significantly larger than the estimate of \cite{noam}. In the case in which $M \gg n$, the estimate provided by Theorem~\ref{main-thm} is asymptotically nearly optimal.

\begin{theorem}
\label{main-thm}
$$
Y_1(M,n) \geq 2 (M-1)^n - 2 (M-2)^n - n (M-2)^{n-1}
$$
\end{theorem}
\begin{proof}
Let $H,f$ be given. Let $X \subseteq [M]^n$ denote the set of weights such that $w(v) = 1$ for exactly one vertex $v$. We will define a bipartite graph $G$, whose left half corresponds to $X$ and whose right half corresponds to $Z_1(H,M,f)$. To avoid confusion between $G$ and $H$, we will refer to the vertices of $G$ as ``nodes.''

Suppose we are given a node $w \in X$ with $w(i) = 1$. We construct edges from $w$ according to three cases:
\begin{enumerate}
\item[(A1)] If $w$ has at least one min-weight edge $e$ such that $i \notin e$, then create an edge from the left-node labeled $w$ to the right-node labeled $w - \chi_e$. As $i \notin e$, note that $w - \chi_e \in [M]^n$.

\item[(A2)] Suppose that $i \in e$ for all min-weight edges $e \in H$. If $w$ is already isolating for $H$, then create an edge from the left-node labeled $w$ to the right-nodes labeled $w, w - \chi_{e - \{i \}}$.

\item[(A3)] Otherwise, suppose that $i \in e$ for all min-weight edges $e$, and there are at least two such edges $e_1, e_2$. Then create edges from the left-node labeled $w$ to the two right-nodes $w - \chi_{e_1 - \{ i \}}$ and $w - \chi_{e_2 - \{i \}}$.
\end{enumerate}

In case (A1), Proposition~\ref{prop1} ensures that the corresponding right-node is isolating. In cases (A2) and (A3), Proposition~\ref{prop1a} ensures that the corresponding right-nodes are isolating.  So all the right-nodes of $G$ with at least one neighbor are isolating. We count such nodes using the following simple identity:
$$
\# \text{right-nodes $u$ with a neighbor} = \sum_{\text{edges $(w,u)$ of $G$}} 1/\text{deg}(u)
$$
For any $w \in X$, we define $R(w)$ as
$$
R(w) = \sum_{\text{edges $(w,u)$ of $G$}} 1/\text{deg}(u).
$$

Thus, we aim to show a lower bound on $\sum_{w \in X} R(w)$.

First, suppose that $w$ falls into case (A1).  Then the resulting right-node $u = w - \chi_{e}$ has exactly one vertex $i$ such that $u(i) = 1$ and $i$ is not contained in the min-weight edge $e$. Thus, $w$ is the sole neighbor of $u$ and $R(w) = 1$.

Next, suppose that $w$ falls into case (A2) or case (A3), and $|w^{-1}(2)| = j$. There are two neighbors of $w$; let us consider one such node $x$, with min-weight edge $e$. 

Let $I$ denote the set of entries $i$ such that $x(i) = 1$. It must be that $I \subseteq e$. If $I = \{ i \}$, then $x$ has at most two neighbors $x$ and $x + \chi_{e - \{i \}}$. 

On the other hand, if $|I| = \{v_1, \dots, v_k \}$ for $k >1$, then $x$ has at most $k$ neighbors $x + \chi_{e - v_1}, \dots, x + \chi_{e - v_k}$. Also observe that $k \leq j+1$. 

Thus, in either case (A2) or (A3), we see that $w$ has two neighbors, and each of these neighbors has degree at most $\max(2, j+1)$. So $R(w) \geq \min(1, \frac{2}{j+1})$. This is also true for case (A1).

Putting all these cases together and summing over $w$:
\begin{align*}
\sum_{w \in X} R(w) &\geq \sum_{w \in X} \min(1, \frac{2}{|w^{-1}(2)| + 1}) = n (M-2)^{n-1} + \frac{2}{j+1} \sum_{j=1}^{n-1} n \binom{n-1}{j} (M-2)^{n-1-j} \\
&= 2 (M-1)^n - 2 (M-2)^n - n (M-2)^{n-1}
\end{align*}
\end{proof}

\begin{corollary}
$$
Y(M,n) \geq 2 (M-1)^n - n \sum_{i=1}^{M-2} i^{n-1}
$$

\end{corollary}
\begin{proof}
We have $Y(M,n) \geq \sum_{j=1}^{M}  Y_1(j,n)  \geq \sum_{j=2}^{M} 2 (j-1)^n - 2 (j-2)^n - n (j-2)^{n-1}$. This telescopes to $2 (M-1)^n - n \sum_{i=1}^{M-2} i^{n-1}$.
\end{proof}

A slight modification of Theorem~\ref{main-thm} can be used when we have an upper bound on the size of an edge of $H$.
\begin{proposition}
Suppose that all the edges in $H$ have cardinality at most $r$, where $r \geq 2$. Then 
$$
|Z_1(H,M,f)| \geq (2/r) n (M-1)^{n-1}
$$
\end{proposition}
\begin{proof}
We construct the same bipartite graph as in Theorem~\ref{main-thm}. However, we will estimate $R(w)$ differently. As in Theorem~\ref{main-thm}, for any right-node $x$, we let $I$ denote the set of entries $i$ with $x(i) = 1$. As before, $x$ has at most $\max(2,|I|)$ neighbors; also, since $I \subseteq e$, we have $|I| \leq r$. So, in case (A2) or case (A3), we have $R(w) \geq \min(1, \frac{2}{r})$. By our assumption that $r \geq 2$, this implies that $R(w) \geq \frac{2}{r}$.

So, summing over $w$:
\begin{align*}
\sum_{w \in X} R(w) &\geq \frac{2}{r} |X| = (2/r) n (M-1)^{n-1} 
\end{align*}

\end{proof}

\section{Graph transformations}
\label{graph-transform-sec}
In this section, we describe certain graph transformations which allow us to reduce the calculation of $Z_1(H,M,f)$ to the behavior of smaller subgraphs. These transformations do not allow us to compute $Z_1(H,M,f)$ in full generality, but they can show certain restrictions on minimal counter-examples to Conjecture~\ref{big-conj2}.

For any hypergraph $H$ and vertex $v \in [n]$, we define by $H-v$ the subgraph induced on the vertices $[n] - \{ v \}$.
\begin{proposition}
\label{induct-prop1}
Suppose $H$ has a vertex $v$ of degree zero. Then
$$
|Z_1(H,M,f)| \geq M |Z_1(H-v, M, f)| + \sum_{j = 2}^M |Z_j(H-v,M,f)|
$$
\end{proposition}
\begin{proof}
For each $w \in Z_1(H - v,M,f)$, we can extend it to $Z_1(H,M,f)$ by assigning any value to $w(v)$. Also, for each $w \in Z_j(H - v, M, f)$ for $j > 1$, we can extend it to $Z_1(H,M,f)$ by assigning $w(v) = 1$.
\end{proof}

\begin{proposition}
\label{induct-prop2}
Suppose that $v \in H$ has degree one (that is, exactly one edge of $H$ contains $v$). Then
$$
|Z_1(H,M,f)| \geq (M-1) |Z_1(H-v,M,f)| + (M-1)^{n-1}
$$
\end{proposition}
\begin{proof}
Suppose without loss of generality that $v = 1$ and that $\tilde e$ is the sole edge containing $v$.

We will construct two classes of isolating weights for $H$. To construct the first class $A_1$, begin with some $w' \in Z_1(H-v,M,f)$. Extend this to $w \in [M]^n$ by assigning some value to $w(1)$. Observe that $w$ will fail to be isolating if and only if the unique  min-weight edge of $H-v$ has the same value as $\tilde e$. Thus, there is at most one value of $w(1)$ such that $w \notin Z_1(H,M,f)$.

First, suppose that there is some choice of $w(1)$ such that $w \notin Z_1(H,M,f)$. In this case, $w'$ extends to $w$ in $M-1$ ways, which are all placed into $A_1$. 

Second, suppose that $w \in Z_1(H,M,f)$ for all $M$ choices of $w(1)$. In this case, we extend $w'$ to $Z_1(H,M,f)$ by assigning values $w(1) = 2, \dots, n$ and placing these into $A_1$. Even though assigning $w(1) = 1$ would also lead to a isolating weight, we do \emph{not} place this into $A_1$.

Thus, each $w' \in Z_1(H-v,M,f)$ corresponds to exactly $M-1$ elements in $A_1$, so that $|A_1| = (M - 1) |Z_1(H-v,M,f)|$

We construct the next class $A_2$ as the image of an injective function $\Psi: [M-1]^{n-1} \rightarrow Z_1(H,M,f)$, as follows. Given $w: \{2, \dots, n \} \rightarrow \{ 2, \dots, M \}$, extend it to $[M]^n$ by assigning $w(1) = 1$. If $\tilde e$ is the unique min-weight edge for $w$, then set $\Psi(w) = w$. Otherwise, let $e \in H - v$ be a min-weight edge for $w$, and let $\Psi(w) = w - \chi_{e}$; by Proposition~\ref{prop1} we have $\Psi(w) \in Z_1(H,M,f)$.

We first claim that $\Psi$ is injective. For, given $w \in \text{image}(\Psi)$, let $e$ denote its unique min-weight edge. If $e = \tilde e$, then $\Psi^{-1}(w) = w$; otherwise $\Psi^{-1}(w) = w + \chi_{e}$.

Next, we claim that $A_2$ is disjoint from $A_1$. For, suppose that $w \in \text{image}(\Psi)$ and $e$ is its unique min-weight edge. As $w \in A_2$ we have $w(1) = 1$.

If $e = \tilde e$, then $w(2) > 1, \dots, w(n) > 1$, so that $\langle w(2), \dots, w(n) \rangle \notin Z_1(H-v, M, f)$. But, for all $x \in A_1$ we have $\langle x(2), \dots, x(n) \rangle \in Z_1(H - v,M,f)$. 

If $e \neq \tilde e$, then observe that $e$ will remain the unique min-weight edge even if we increment $w(1)$ from its initial value of $1$ to an arbitrary value. Thus, even if we had started with $\langle w(2), \dots, w(n) \rangle \in Z_1(H - v,M,f)$ to construct an element of $A_1$, we would not have been allowed to assign $w(1) = 1$. Thus, $w \notin A_1$.

Thus, we see that $|A_2| = (M-1)^{n-1}$ and $A_2$ is disjoint from $A_1$. So $|Z_1(H,M,f)| \geq |A_1| + |A_2| = (M-1) |Z_1(H-v,M,f)| + (M-1)^{n-1}$.
\end{proof}

\begin{corollary}
Suppose that $H$ is $1$-degenerate. Then $|Z_1(H,M,f)| \geq n (M-1)^{n-1}$
\end{corollary}
\begin{proof}
Let $v_1, \dots, v_n$ be an ordering of the vertices such that $v_i$ has degree $\leq 1$ in $H[v_i, \dots, v_n]$. By induction for $i = n, \dots, 1$, observe that $|Z_1(H[v_1, \dots, v_n])| \geq (n-i) (M-1)^{n-i-1}$; the inductive step follows from Propositions~\ref{induct-prop1}, \ref{induct-prop2}.
\end{proof}

\begin{proposition}
Suppose that $H_1$ is a hypergraph on vertex set $V_1$ and $H_2$ is a hypergraph on vertex set $V_2$, where $V_1, V_2$ are disjoint.
Then
$$
|Z_1(H_1 \sqcup H_2,M,f)| \geq (M-1)^{|V_2|} |Z_1(H_1, M,f)| + (M-1)^{|V_1|} |Z_1(H_2, M, f)|
$$
\end{proposition}
\begin{proof}
Let $n_1 = |V_1|, n_2 = |V_2|, n = n_1 + n_2$. Suppose without loss of generality that $V_1 = \{1, \dots, n_1 \}$ and $V_2 = \{n_1 + 1, \dots, n \}$ and let $H = H_1 \sqcup H_2$.
We will construct two classes of isolating weights for $H$. 

The first class is constructed as the image of an injective function $\Psi_1: Z_1(H_1, M, f) \times \{2, \dots, M \}^{n_2} \rightarrow Z_1(H,M,f)$ as follows. Given $u \in Z_1(H_1, M, f)$ and $v \in [M-1]^{n_2}$, define $w \in [M]^n$ by $w = \langle u(1), \dots, u(n_1), v(1), \dots, v(n_2) \rangle$. Suppose that $w$ has some min-weight edge $e \in H_2$; in this case, define $\Psi_1(u, v) = w - \chi_e$. If $w$ has no min-weight edges from $H_2$, then as $u$ is isolating for $H_1$, necessarily $w$ is isolating for $H$, and we define $\Psi_1(u,v) = w$.

We claim that $\Psi_1$ is injective. For, given $w = \langle u(1), \dots, u(n_1), v(1), \dots, v(n_2) \rangle \in \text{image}(\Psi_1)$, let $e$ be its unique min-weight edge. If $e \in H_2$, then $\Psi_1^{-1}(u,v) = (u, v + \chi_e)$; otherwise, if $e \in H_1$, then $\Psi_1^{-1}(u,v) = (u,v)$.

We define $\Psi_2: \{2, \dots, M \}^{n_1} \times Z_1(H_2, M, f) \rightarrow Z_1(H,M,f)$ in the same fashion, interchanging the roles of $H_1$ and $H_2$. 

We now claim that the images of $\Psi_1$ and $\Psi_2$ are disjoint. For, suppose that $(u,v)$ is simultaneously in the image of $\Psi_1$ and $\Psi_2$. Let $e_2$ be its unique min-weight edge in $H$; suppose without loss of generality that $e_2 \in H_2$. 

So $\Psi_1^{-1}(u,v) = (u, v + \chi_{e_2})$. In particular, $u \in Z_1(H_1,M,f)$ so $L(u) = 1$. Also, we have $\Psi_2^{-1}(u,v) = (u,v)$. In particular, $u \in \{2, \dots, M \}^{n_1}$ so $L(u) > 1$. This is a contradiction.

Thus, the images of $\Psi_1$ and $\Psi_2$ are disjoint so
\begin{align*}
|Z_1(H, M,f)| &\geq | \text{image} (\Psi_1) | + | \text{image}(\Psi_2) | = |Z_1(H_1, M,f)| (M-1)^{n_2} + |Z_1(H_2, M, f)| (M-1)^{n_1}
\end{align*}
\end{proof}

\begin{corollary}
Suppose that $H$ is a counter-example to Conjecture~\ref{big-conj2}, and among all such counter-examples it minimizes the number of vertices $n$. Then $H$ is connected and all the vertices of $H$ have degree strictly greater than $1$.
\end{corollary}

\section{The case of $M = 2$}
\label{m2sec}
In this section, we will prove Conjecture~\ref{big-conj2} for $M = 2$. The basic idea of this proof is to identify a class of isolating weights which we refer to as \emph{special isolating weights.} We will show that there are at least $n (M-1)^{n-1} = n$ special isolating weights. See \cite{faber-prev} for a more detailed analysis of structural properties of the isolating weights in this case.

\begin{definition}
Suppose $H$ is a non-empty hypergraph. We say $w: [n] \rightarrow [2]$ is a \emph{special isolating weight} for $H$ if there is an edge $e$ satisfying the following conditions:
\begin{enumerate}
\item For all $i \in e, j \notin e$ we have $w(i) \leq w(j)$.
\item For all $e' \neq e, e' \in H$ we have $w(e') > w(e)$.
\end{enumerate}
\end{definition}

The objective function $f$ plays no part in this definition. We define $Z'(H)$ to be the set of special isolating weights for $H$. (If $H$ contains no edges, then we define $Z'(H) = [2]^n$). The following key result shows why special isolating weights are simpler to deal with:
\begin{proposition}
\label{reduce-prop}
Let $H$ be a hypergraph, and let $r$ denote the the minimum cardinality of the edges of $H$. Let $H_r$ denote the subgraph of $H$ consisting of the edges of cardinality exactly $r$.

Then, for any objective function $f$, we have
$$
Z'(H_r) \subseteq Z(H, 2, f)
$$
\end{proposition}
\begin{proof}
Let $w \in Z'(H_r)$, with min-weight edge $e$. We will show that $e$ remains the unique min-weight edge for $fw$ in $H$. 

It is either the case that $w(i) = 1$ for all $i \in e$, or $w(i) = 2$ for all $i \notin e$.  The proofs are similar so we only deal with the first case.

First, consider some other edge $e'$ of cardinality $r$. Then by hypothesis $e'$ contains a point $i$ with $w(i) = 2$ so that $fw(e') \geq f(2) + (r-1) f(1) > r f(1) = fw(e)$. 

Next, suppose $e'$ has cardinality strictly greater than $r$. Then
\begin{align*}
fw(e') - fw(e) &= \sum_{i \in e' - e} fw(i) - \sum_{i \in e - e'} fw(i)  \geq \sum_{i \in e' - e} f(1) - \sum_{i \in e - e'} f(1) = f(1) ( |e'| - |e| ) > 0
\end{align*}
\end{proof}

\begin{proposition}
\label{m2prop1}
Suppose every edge of $H$ has cardinality exactly $r$. Then $|Z'(H)| \geq n$.
\end{proposition}

\begin{proof}
We will suppose for this proof that $r \leq n/2$; when $r > n/2$, then the same argument applies, interchanging the roles of $e$ and $[n] - e$.

Suppose that $H$ has $m$ edges. For an edge $e \in H$, let $H_1(e) = \{ e - e' \mid e' \in H \}$ and let $H_2(e) = \{e' - e \mid e' \in H \}$.  Let $C_1(e)$ be a minimum vertex cover of $H_1(e)$ and $C_2(e)$ be a minimum vertex cover of $H_2(e)$. We define $S(e)$ to be the number of special isolating edges whose min-weight edge is $e$.

There are two types of special isolating weights we can form with min-weight edge $e$: we may assign $w(i) = 1$ for $i \in e$ and $w(i) = 2$ for $i \in C_2(e)$ and $w(i)$ arbitrary otherwise; or we may assign $w(i) = 2$ for $i \notin e$ and $w(i) = 1$ for $i \in C_1(e)$ and $w(i)$ arbitrary otherwise. There is an overlap between these classes if we set $w(i) = 1$ for $i \in e$ and $w(i) = 2$ for $i \notin e$. Thus (taking into account double-counting), we have $S(e) \geq 2^{n-r - |C_2(e)|} + 2^{r - |C_1(e)|} - 1$.

Now, clearly $|C_1(e)| \leq m-1$ and $|C_2(e)| \leq m - 1$ (we may select one vertex from each of the other edges). Thus, we have $S(e) \geq 2^{n-r-m+1}$ and hence $|Z'(H)| \geq m 2^{n-r-m+1}$. If $m \leq n - r$, then simple calculus shows that this is at least $2 (n - r) \geq n$, and we are done.

Also, observe that $|C_1(e)| \leq r$ and $|C_2(e)| \leq n - r$ ($e$ is a vertex cover of $H_1(e)$ and $[n] - e$ is a vertex cover of $H_2(e)$). So $S(e) \geq 1$. If $m \geq n$, then $|Z'(H)| \geq m \geq n$ and we are again done. 

So, let us suppose that $n-r < m < n$.  We would like to show that there are many edges that have the property $|C_2(e)| < n - r$ or $|C_1(e)| < r$. Such edges will have $S(e) \geq 2$. We say that such edges are \emph{rich}. If there are $a$ rich edges, then
$$
|Z'(H)| \geq 2 a + (m-a)
$$

Now consider any non-rich edge $e$,  that is $|C_2(e)| = n-r$ and $|C_1(e)| = r$. Consider any $v \in [n]-e$. Since the set $[n] - e - \{v \}$  is not a vertex cover of $H_2(e)$, it must be that $H_2(e)$ contains a singleton edge $\{v \}$. (Note that $H_2(e)$ cannot contain the edge $\emptyset$.) As every edge of $H$ has cardinality $r$, this in turn implies that $H$ contains an edge obtained by swapping a single element of $e$ with $v$, that is, an edge of the form $e \oplus \{v, u \}$ where $u \in e$.  Similarly, in order to have $|C_1(e)| = r$, then for each $v \in e$ there must be an edge of the form $e \oplus \{v, u \}$ where $u \in [n] - e$.

If all the edges are rich, then $|Z'(H)| \geq 2 m \geq 2 (n-r) \geq n$ and we are done. So fix some non-rich edge $\tilde e \in H$. For each $i \in \tilde e, j \notin \tilde e$, define the indicator variable $K_{ij}$ which is equal to one if $\tilde e \oplus \{i, j \} \in H$, and zero otherwise. We have shown that $\forall i \sum_j K_{ij} \geq 1, \forall j \sum_i K_{ij} \geq 1$.  Also, observe that $m \geq 1 + \sum_{i,j} K_{i,j}$.

Define the set $L = \{j \in [n] - \tilde e \mid \sum_i K_{ij} = 1 \}$.  For each $j \in L$, let $e_j$ be the (unique) edge  of the form $e_j = \tilde e \oplus \{i_j, j \}$ with $i_j \in \tilde e$. We have $\sum_i K_{i,j} \geq 1$ for all $j \in [n] - \tilde e$ and $\sum_i K_{i,j} \geq 2$ for $j \in [n] - \tilde e - L$. Summing over $j \in [n] - \tilde e$ gives $\sum_{i,j} K_{i,j} \geq |L| + 2 (n - r - |L|)$. As $\sum_{i,j} K_{i,j} \leq m - 1$, we have $$
|L| \geq 2 (n-r) - m + 1
$$

We next claim that for each $j \in L$ the edge $e_j$ is rich.  For, consider an edge of the form $e' = \tilde e \oplus \{i', j' \}$ with $i' \neq i_j$. So $e' \oplus e_j = \{i_j, j \} \oplus \{i', j' \}$. If $j = j'$ then there would be \emph{two} edges obtained from $\tilde e$ by swapping $j$, contradicting that $j \in L$. Thus, $j \neq j'$ and so $|e_j \oplus e'| = 4$. Thus, $e'$ is not equal to $e_j$ or a swap of $e_j$ (as in those cases we would have $|e_j \oplus e'| \leq 2$.)

So there are at least $1 + \sum_{i' \in e - \{i_j \} } \sum_{j' \notin e} K_{i' j'} \geq r$ edges which are not swaps of $e_j$. But, in order for $e_j$ to be non-rich, there must at least $n-r$ edges obtained by swaps of $e_j$. So, a necessary condition for $e_j$ to be non-rich is $m-r \geq n-r$; this contradicts our assumption that $m < n$.

Thus for each $j \in L$ there is an rich edge $e_j$. Furthermore, if $j \neq j'$ then $e_j \neq e_{j'}$, and so $a \geq |L| \geq 2 (n -r) - m + 1$. This implies $|Z'(H)| \geq 2 a + (m-a) \geq 2 (n-r) + 1 \geq n+1$.
\end{proof}

\begin{corollary}
$$
Y_1(2,n) \geq n
$$
\end{corollary}
\begin{proof}
Consider any hypergraph $H$ and objective function $f$. Let $r$ denote the minimum edge size of $H$ and let $H_r$ denote the set of edges of $H$ with cardinality $r$.

By Proposition~\ref{reduce-prop} we have $Z(H,2,f) \supseteq Z'(H_r)$.

Suppose that $H_r$ contains more than one edge. By Proposition~\ref{m2prop1} we have $|Z'(H_r)| \geq n$. Also, note that $\langle 2,2, \dots, 2 \rangle \notin Z'(H_r)$ (since this weight would cause all edges of $H_r$ to be min-weight) and so $Z'(H_r) \subseteq Z_1(H,2,f)$.

Suppose that $H_r$ contains a single edge. Then one can easily see that $| Z'(H_r) | = 2^r + 2^{n-r} - 1 \geq 2^{1 + n/2} - 1$. Hence $|Z(H,2,f)| \geq 2^{1 + n/2} - 1$ and $|Z_1(H,2,f)| \geq 2^{1 + n/2} - 2$; this is at least $n$ for $n \geq 2$.
\end{proof}

\section{Linear hypergraphs}
\label{linearsec}
A \emph{linear hypergraph} $H$ is a hypergraph with the property that $| e \cap e' | \leq 1$ for any distinct edges $e, e' \in H$. An ordinary graph is a linear hypergraph. In this section, we prove that Conjecture~\ref{big-conj2} holds for linear hypergraphs.

\begin{definition}
For any edge $e \subseteq [n]$ and $i \in e$, we define the \emph{next vertex} of $e$ as follows. If there is some vertex $j \in e$ such that $j > i$, then $\text{Next}(i,e)$ is defined to be the smallest such $j$. Otherwise, if $i$ is the largest element of $e$, then we define $\text{Next}(i,e)$ to be the \emph{smallest} element of $e$.

\end{definition}

Recall that $X$ is the set of weights $w$ such that $w(i) = 1$ for exactly one $i \in [n]$.

\begin{proposition}
\label{gprop}
Suppose $w \in X$ satisfies $w(i) = 1$, and all min-weight edges under $w$ contain vertex $i$ and have cardinality strictly greater than $1$. For any $e$, define 
$$
g(w,e) = w - \chi_{ \{ \text{Next}(i,e) \}}.
$$

If any edge $e$ is a min-weight edge of $w$, then $e$ is the unique min-weight edge for $g(w, e)$. 
\end{proposition}
\begin{proof}
Let $j = \text{Next}(i,e), w' = g(w,e)$ and let $e' \in H$ be another edge. If $j \in e'$, then as $H$ is linear $i \not \in e'$, and so $fw(e') > fw(e)$. Both $e, e'$ contain vertex $j$ so $fw'(e') > fw'(e)$.

Otherwise, suppose $j \notin e'$. Then $fw'(e) < fw(e) \leq fw(e') = fw'(e')$.
\end{proof}

\begin{proposition}
\label{linear-prop}
Let $H$ be a linear hypergraph all of whose edges have cardinality at least two. Then for any objective function $f$ we have
$$
|Z_1(H,M,f)| \geq n (M-1)^{n-1}
$$
\end{proposition}
\begin{proof}
As in Theorem~\ref{main-thm}, we will construct a bipartite graph $G$, whose left half corresponds to $X$ and whose right half corresponds to $Z_1(H,M,f)$. Suppose $w \in X$ has $w(i) = 1$. We construct edges from $w$ according to three cases:
\begin{enumerate}
\item[(B1)] If $w$ has at least one min-weight edge $e$ with $i \notin e$, then create an edge from the left-node labeled $w$ to the right-node labeled $w - \chi_e$. As $i \notin e$, note that $w - \chi_e \in [M]^n$. 

\item[(B2)] Suppose that $i \in e$ for all min-weight edges $e$. If $w$ is already isolating for $H$ with min-weight edge $e$, then create edges from the left-node labeled $w$ to the two right-nodes labeled $w, g(w,e)$.

\item[(B3)] Otherwise, suppose that $i \in e$ for all min-weight edges $e$, and there are at least two such edges $e_1, e_2$. Then create edges from the left-node labeled $w$ to the two right-nodes $g(w,e_1), g(w,e_2)$.
\end{enumerate}

In case (B1), Proposition~\ref{prop1} ensures that the corresponding right-node is isolating. In cases (B2) and (B3), Proposition~\ref{gprop} ensures that the corresponding right-nodes are isolating.  So all the right-nodes which have at least one neighbor are isolating. We again use the identity
$$
\# \text{right-nodes $u$ with a neighbor} = \sum_{\text{edges $(w,u)$}} 1/\text{deg}(u)
$$
and for $w \in X$ we define $R(w) = \sum_{\text{edges $(w,u)$}} 1/\text{deg}(u)$.

Now consider some right-node $x$, with a unique min-weight edge $e$. We examine the potential ways in which $x$ can have a neighbor.

If there is $i \notin e$ with $x(i) = 1$ then necessarily $x$ has only a single neighbor $w = x + \chi_e$ coming from case (B1).

So, suppose that $x(i) > 1$ for all $i \notin e$. Let $I$ denote the set of entries $i \in e$ with $x(i) = 1$. Since $x$ could only have a neighbor from cases (B2) or (B3), it must be that $1 \leq |I| \leq 2$.

If $|I| = \{ i \}$, then the neighbors of $x$ could arise either when $w = x$ and case (B2) occurred or $w = x + \chi_{ \{ \text{Next}(i,e) \}}$ and (B2) or (B3) occurred.

If $|I| = \{i_1, i_2 \}$, then the only possible neighbors of $x$ are $w_1 = x + \chi_{ \{i_1 \}}$ and $w_2 = x + \chi_{ \{ i_2 \}}$.

Now, consider some left-node $w$. We see that in case (B1),  $w$ has a single neighbor $x$, which in turn has only a single neighbor $w$. So $R(w) = 1$. In case (B2) or (B3), then $w$ has two neighbors, each of which has at most $2$ neighbors, so $R(w) \geq 1$. Putting all these cases together and summing over $w$:
$$
\sum_{w \in X} R(w) \geq 1 \times |X| = n (M-1)^{n-1}
$$
\end{proof}

\begin{corollary}
Suppose that $H$ is a linear hypergraph. Then $|Z_1(H,M,f)| \geq n (M-1)^{n-1}$
\end{corollary}
\begin{proof}
We prove this by induction on $n$.  If all the edges of $H$ have cardinality $> 1$, then this follows from Proposition~\ref{linear-prop}. Otherwise, let $\{v \}$ be a singleton edge in $H$. We may assume that $H$ contains no other edge containing $v$. Then observe that $H - v$ is a linear hypergraph on $n - 1$ vertices and by induction hypothesis $|Z_1(H-v,M,f)| \geq (n-1) (M-1)^{n-2}$. By Proposition~\ref{induct-prop2} 
$$
|Z_1(H,M,f)| \geq (M-1) |Z_1(H-v,M,f)| + (M-1)^{n-1} \geq n (M-1)^{n-1}
$$

\end{proof}

\section{Algorithmic applications and asymptotics}
\label{sec:algorithmic}
As we have discussed, the main use of the Isolation Lemma in the context of algorithms is the following: we have a hypergraph $H$ (which may not be presented explicitly), and we wish to find some $w: [n] \rightarrow [M]$ such that $w$ is isolating on $H$, where $M$ is as small as possible. Since we do not have access to $H$ in any convenient way, the usual way to find $w$ is to simply choose one from $[M]^n$ uniformly at random. When we do so, the resulting $w$ is isolating with probability $p = |Z(H, M, f)|/M^n$.

In these settings, we will typically have $n \rightarrow \infty$ and $M \geq n$, and we make the following useful estimate for $p$:
\begin{proposition}
\label{p-prop}
Let $\phi = n/M$. Define
$$
h_1(\phi) = \frac{\phi}{e^{\phi} - 1}, \qquad h_2(\phi) = \frac{2 (e^{\phi} - 1) - \phi}{e^{\phi} (e^{\phi} - 1)}.
$$

We have $p \geq h_2(\phi) - O(1/M)$. Furthermore, if Conjecture~\ref{big-conj2} holds, then $p \geq h_1(\phi) - O(1/M)$.
\end{proposition}

By contrast, using the cruder bound $|Z(H,M,f)| \geq (M-1)^n$, we would be able to show only that 
$$
p \geq h_0(\phi) - O(1/M) \qquad \text{where $h_0(\phi) = e^{-\phi}$}.
$$

In light of our analysis in terms of layers, we propose a slightly different method for selecting $w$. Instead of selecting $w$ uniformly from $[M]^n$, suppose we instead select $w$ uniformly from $[M]^n - \{2, \dots, n \}^n$. In this case, the resulting $w$ is isolating with probability
$$
q = \frac{|Z_1(H,M,f)|}{M^n - (M-1)^n }.
$$
We bound $q$ using either Theorem~\ref{main-thm} or Conjecture~\ref{big-conj2}. Note that these estimates avoid the $O(1/M)$ error term of Proposition~\ref{p-prop}. 
\begin{proposition}
Let $\phi = n/M \leq 1$. Then $q \geq h_2(\phi)$. Furthermore, if Conjecture~\ref{big-conj2} holds, then $q \geq h_1(\phi)$.
\end{proposition}

In the limit as $\phi \rightarrow 0$, we have simpler estimates:
\begin{corollary}
We have
\begin{align*}
h_0(\phi) &= 1 - \phi + O(\phi^2) \\
h_1(\phi) &= 1 - \frac{\phi}{2} + \frac{\phi^2}{12} - O(\phi^4) \\
h_2(\phi) &= 1 - \frac{\phi}{2} - \frac{\phi^2}{12} + O(\phi^3)
\end{align*}
\end{corollary}

Thus, the estimates provided by Theorem~\ref{main-thm} and Conjecture~\ref{big-conj2} are asymptotically equivalent (up to second order) for $\phi \rightarrow 0$, and improve by a factor of roughly 2 over the estimate of \cite{noam}. From an algorithmic point of view, this means that, in order for an algorithm to achieve a given high success probability (i.e. small probability that $w$ fails to be isolating), we need roughly one less bit of accuracy in the size of the weights as compared to the estimate of \cite{noam}.

\subsection{Allowing zero-weight vertices}
In our definition of the objective function $f$, we have restricted the range of $f$ to be \emph{strictly positive} real numbers. In some algorithmic applications, zero-weight vertices have been allowed \cite{klivans-spielman}. It is natural to ask what bounds on $|Z(H,M,f)|$ can be shown when the function $f$ is allowed to take on the value zero. Let us define the quantity $Y'(M,n)$ as
$$
Y'(M,n) = \min_{H,f} |Z(H,M,f)|
$$
where $H$ ranges over all hypergraphs on $n$ vertices and $f$ ranges over all functions $f: [M] \rightarrow R_{\geq 0}$. Note that in this case, we can no longer assume without loss of generality that $H$ is inclusion-free.

In this setting, the bound of \cite{noam} is exactly tight.
\begin{proposition}
\label{zzprop}
We have 
$$
Y'(M,n) = (M-1)^n
$$
\end{proposition}
\begin{proof}
First, we show that for all $M, H, f$ we have $|Z(H,M,f)| \geq (M-1)^n$, via a slight modification of Proposition~\ref{prop1}. We construct an injective map $\Psi$ from $\{2, \dots, M \}^n$ to $Z(H,M,f)$, as follows. Given any  $w \in \{2, \dots, M \}^n$, we let $E_w$ denote the set of min-weight edges. Arbitrarily select some $e \in E_w$ which is \emph{inclusion-wise maximal}; that is, there is not any other $e' \in E_w$ with $e \subsetneq e'$. Then set $\Psi(w) = w - \chi_e$. One can easily verify that $\Psi(w)$ is isolating.

Next, we construct a hypergraph $H$ with $|Z(H,M,f)| \leq (M-1)^n$. Let $H$ be the full power-set of $n$ elements and define $f: \{1, \dots, M \} \rightarrow \mathbf R_{\geq 0}$ by $f(i) = i-1$. Observe that if $w(i) = 1$ for any $i \in [n]$, then $\emptyset, \{ i \}$ are both min-weight edges, and so $w$ is not isolating. So an isolating weight $w$ must have $w(i) \in \{ 2, \dots, M \}$ and so $|Z(H,M,f)| \leq (M-1)^n$.
\end{proof}

We emphasize that for most application of the isolation method, one can \emph{choose} the objective function $f$ in order to maximize $|Z(H,M,f)|$. Thus, Proposition~\ref{zzprop} shows that it is more efficient to choose the range of $f$ to be strictly positive.

\section{Further problems}
\label{sec:conclusion}
In addition to the main Conjecture, there are several other interesting questions one may ask:
\begin{enumerate}
\item Are there any simple graph parameters (such as edge cardinality, number of edges, etc.) such that $Z(H)$ or $Z_1(H)$ is significantly larger than our conjectured lower bound?
\item One may extend the type of objective functions, for example, one may allow distinct functions $f_i$ for each vertex $i$ (and so the value of an edge $e$ is $\sum_{i \in e} f_i(w_i)$). This generalization is needed in \cite{narayanan}, for instance. One may even further extend the objective function to be non-linear. Do similar bounds apply?
\item We have seen that there is a higher probability that $w$ is isolating if $w$ is forced to contain at least one entry of value $1$. Are there any other restrictions that we may place on $w$ to increase this probability (without taking advantage of knowledge of $H$)? We conjecture that this is not the case, i.e. if $X$ is any subset of $[M]^n$ then we have
$$
\min_{H,f} \frac{  |Z(H, M, f) \cap X| }{|X|} \geq \frac{ n (M-1)^{n-1} }{M^n - (M-1)^{n-1}}
$$
  
\end{enumerate}

\section{Acknowledgments}
Thanks to Noah Streib for suggestions on the $M=2$ proof strategy. Thanks to Aravind Srinivasan for some helpful discussions.

\end{document}